\documentclass[11pt,reqno]{amsart}

\usepackage{mathtools,amssymb,amsthm,mathrsfs}
\usepackage{paralist}
\usepackage{microtype}
\usepackage[T1]{fontenc}
\usepackage[utf8]{inputenc}
\usepackage[colorlinks=true,linkcolor=blue,citecolor=blue,urlcolor=blue]{hyperref}

\newtheorem{theorem}{Theorem}[section]
\newtheorem{proposition}[theorem]{Proposition}
\newtheorem{lemma}[theorem]{Lemma}

\theoremstyle{definition}

\theoremstyle{remark}
\newtheorem{remark}[theorem]{Remark}

\newcommand\Set[2]{\left\{#1\ \middle\vert\ #2 \right\}}

\makeatletter
\@namedef{subjclassname@2020}{\textup{2020} Mathematics Subject Classification}
\makeatother

\title[Gromov-compactifiability of Wasserstein spaces]{Gromov-compactifiability of Wasserstein spaces}
\author{Huajian Jiang}
\address{School of Mathematics, Sichuan University, Chengdu 610064, China}
\email{huajianjiang@scu.edu.cn}
\author{Zhi-Xiang Zhu}
\address{School of Mathematics, Nanjing University, Nanjing 210093, China}
\email{zhuzhixiang0928@gmail.com}

\begin{document}

\begin{abstract}
We characterize Gromov-compactifiability of Wasserstein spaces over Polish metric spaces for \(1\leq p\le\infty\). The space \(\mathcal P_1(X)\) is Gromov-compactifiable if and only if \(X\) is bounded and Gromov-compactifiable, whereas \(\mathcal P_p(X)\) is Gromov-compactifiable for every \(1<p<\infty\). At the other endpoint, \(\mathcal P_\infty(X)\) is Gromov-compactifiable if and only if \(X\) is. For \(p=1\), the sufficiency argument combines an averaged triangle defect with tightness. For \(1<p<\infty\), the main tool is the uniform convexity of \(L^p\). The \(p=\infty\) proof uses the finite-set criterion and a mass redistribution construction.
\end{abstract}

\maketitle

\section{Introduction}

Let \((X,d)\) be a metric space and fix \(x_0\in X\). The horofunction map is
\[
h_{x_0}\colon X\to\mathbb R^X,
\qquad
h_{x_0}(z)(x):=d(x,z)-d(x_0,z).
\]
The closure of its image in the topology of pointwise convergence is the horofunction extension \(\overline X^{\rm h}\). Following \cite{MR4952998}, we call \(X\) \emph{Gromov-compactifiable} if \(h_{x_0}\) is a topological embedding, or, equivalently, if \(\overline X^{\rm h}\) is a compactification of \(X\). This property does not depend on the choice of \(x_0\). Although \(\overline X^{\rm h}\) is always compact in the pointwise topology, the topology induced on \(h_{x_0}(X)\) need not agree with the metric topology of \(X\).

Daniilidis, Garrido, Jaramillo, and Tapia-Garc\'ia characterized this property by a finite-set condition that records a uniform defect in the triangle inequality \cite{MR4952998}. For Banach spaces, they also related Gromov-compactifiability to non-octahedrality of the norm. The finite-set condition is the main link between the geometry of \(X\) and that of its Wasserstein spaces.

For \(1\leq p<\infty\), let \(\mathcal P_p(X)\) denote the space of Borel probability measures on \(X\) with finite \(p\)-th moment, equipped with \(W_p\). Let \(\mathcal P_\infty(X)\) denote the space of probability measures with bounded support, equipped with \(W_\infty\). We refer to \cite{MR2401600,MR2459454} for background. When \(X\) is noncompact, its Wasserstein spaces need not be locally compact, so properness does not settle the question considered here.

The horofunction construction originates in Gromov's work \cite{MR919829}. Related metric compactifications were studied by Rieffel for locally compact metric spaces \cite{Rieffel2002}, by Walsh for finite-dimensional normed spaces \cite{Walsh2007}, and by Guti\'errez for infinite-dimensional \(\ell_p\) and \(L_p\) spaces \cite{Gutierrez2019,Gutierrez2020}. Throughout the paper, ``Gromov-compactifiable'' has the meaning introduced in \cite{MR4952998}. The phrase ``Gromov compactification'' is also used for other constructions in \cite{MR4700367,MR4205705,MR4422057}.

Several papers study asymptotic functions on Wasserstein spaces from a different viewpoint. Bertrand and Kloeckner described the geodesic boundary of \(\mathcal P_2(X)\) over locally compact Hadamard spaces \cite{BertrandKloeckner2012}. Zhu, Wu, and Cui studied horofunctions generated by sequences of atomic measures \cite{MR4154570}, while Zhu, Li, and Cui studied rays, co-rays, and Busemann functions in \(\mathcal P_p(X)\) for \(p>1\) over noncompact locally compact length spaces \cite{MR4249875}. These works concern particular boundary functions or geodesic behavior. The present question is topological. We ask when the horofunction map recovers the Wasserstein topology, without a local compactness or length-space assumption on \(X\).

Our results divide according to the value of \(p\). The first theorem settles the remaining direction for \(p=1\).

\begin{theorem}\label{thm:p1_compactifiable}
Let \((X,d)\) be a Polish metric space. Then \(\mathcal P_1(X)\) is Gromov-compactifiable if and only if \(X\) is bounded and Gromov-compactifiable.
\end{theorem}

Daniilidis et al.\ proved that Gromov-compactifiability of \(\mathcal P_1(X)\) implies that of \(X\) \cite[Proposition~3.16(ii)]{MR4952998}, and they asked whether the converse holds \cite[p.~1153]{MR4952998}. Jiang and Cui later showed that unboundedness is an obstruction \cite[Theorem~1.1]{MR5096390}. We prove the remaining sufficiency statement for bounded Gromov-compactifiable spaces. The linear transport cost allows us to average the triangle defect along an optimal coupling, while tightness yields a finite family of witnesses in \(\mathcal P_1(X)\).

For interior exponents, no assumption beyond Polishness is required.

\begin{theorem}\label{thm:compactifiable_P_p}
Let \((X,d)\) be a Polish metric space and let \(p\in(1,\infty)\). Then \(\mathcal P_p(X)\) is Gromov-compactifiable.
\end{theorem}

The proof uses the uniform convexity of \(L^p\). Horofunction convergence produces asymptotic equality in Minkowski's inequality along optimal couplings, and Clarkson's inequalities show that the normalized distance functions converge to one another in \(L^p\). A comparison at two base points shows that the candidate limit is a Dirac measure. A two-point test measure then rules out a positive Wasserstein distance from that limit.

At the other endpoint, the answer is again governed by the base space.

\begin{theorem}\label{thm3}
Let \((X,d)\) be a Polish metric space. Then \(\mathcal P_\infty(X)\) is Gromov-compactifiable if and only if \(X\) is Gromov-compactifiable.
\end{theorem}

If \(X\) has at least two points, then \((\mathcal P_\infty(X),W_\infty)\) is nonseparable. We therefore use the finite-set condition rather than a sequential argument. For each witness in \(X\), we construct a measure by moving a small portion of the base measure to that witness. A radial mass estimate and a coupling adapted to the resulting measures provide the required triangle defect.

Thus the interior range \(1<p<\infty\) differs sharply from the two endpoints. At \(p=1\), boundedness and Gromov-compactifiability of \(X\) are both needed, while at \(p=\infty\) only the latter remains.

Section~\ref{S5} records the corresponding atom criterion for real \(L^\infty\) spaces. For the localizable case, the criterion follows by combining \cite[Corollary~6.7]{CiaciLangemetsLissitsin2022} with \cite[Theorem~3.4]{MR4952998}. Our proof uses the finite-set condition and applies to arbitrary measure spaces.

The paper is organized as follows. Section~\ref{S1} contains the preliminaries. Sections~\ref{S2}--\ref{S4} prove Theorems~\ref{thm:p1_compactifiable}--\ref{thm3}, and Section~\ref{S5} treats \(L^\infty\).

\section{Preliminaries}\label{S1}

\subsection{Horofunction extension}
Fix a metric space \((X,d)\) and \(x_0\in X\). Let
\[
\operatorname{Lip}^1_{x_0}(X)
:=
\Set{f\colon X\to\mathbb R}
{f(x_0)=0,\ \operatorname{Lip}(f)\leq 1},
\]
endowed with the topology of pointwise convergence. The horofunction map
\[
h_{x_0}\colon X\to \operatorname{Lip}^1_{x_0}(X),
\qquad
h_{x_0}(z)(x):=d(x,z)-d(x_0,z),
\]
is continuous and injective. We write \(h_z:=h_{x_0}(z)\) for \(z\in X\). The horofunction extension \(\overline X^{\rm h}\) is the closure of \(h_{x_0}(X)\) in \(\operatorname{Lip}^1_{x_0}(X)\). Up to the canonical identification, it does not depend on the choice of \(x_0\) \cite{MR4952998}.

For every \(f\in\operatorname{Lip}^1_{x_0}(X)\),
\[
|f(x)|\leq d(x,x_0)\qquad\text{for all }x\in X,
\]
so
\[
\operatorname{Lip}^1_{x_0}(X)
\subset
\prod_{x\in X}[-d(x,x_0),d(x,x_0)].
\]
The set on the left is closed in the product topology, and the product is compact by Tychonoff's theorem. Hence \(\overline X^{\rm h}\) is compact in the topology of pointwise convergence. The question is whether \(h_{x_0}\) induces the original topology on its image. Equivalently, \(X\) is Gromov-compactifiable if and only if \(h_{x_0}\) is a topological embedding.

We first record a convergence criterion for Gromov-compactifiability; later we shall use the finite-set characterization from \cite{MR4952998}.

\begin{proposition}\label{pro:net_convergence}
Let \((X,d)\) be a metric space.
\begin{enumerate}[\rm (i)]
\item \(X\) is Gromov-compactifiable if and only if, for every \(x\in X\) and every net \((z_i)_{i\in I}\) in \(X\), the condition
\[
\lim_{i\in I}\big(d(y,z_i)-d(x,z_i)\big)=d(y,x)
\qquad\text{for every }y\in X
\]
implies
\[
\lim_{i\in I}d(z_i,x)=0.
\]

\item If \(X\) is separable, then \(X\) is Gromov-compactifiable if and only if, for every \(x\in X\) and every sequence \((z_n)\) in \(X\), the condition
\[
\lim_{n\to\infty}\big(d(y,z_n)-d(x,z_n)\big)=d(y,x)
\qquad\text{for every }y\in X
\]
implies
\[
\lim_{n\to\infty}d(z_n,x)=0.
\]
\end{enumerate}
\end{proposition}

\begin{proof}
Fix \(x\in X\) and a net \((z_i)_{i\in I}\) in \(X\). Suppose first that
\[
\lim_{i\in I}\big(d(y,z_i)-d(x,z_i)\big)=d(y,x)
\qquad\text{for every }y\in X.
\]
Then, for every \(y\in X\),
\begin{align*}
h_{z_i}(y)-h_x(y)
&=
\big(d(y,z_i)-d(x,z_i)-d(y,x)\big)\\
&\qquad-
\big(d(x_0,z_i)-d(x,z_i)-d(x_0,x)\big)
\to0.
\end{align*}
Thus \(h_{z_i}\to h_x\) pointwise.

Conversely, if \(h_{z_i}\to h_x\) pointwise, then, for every \(y\in X\),
\[
d(y,z_i)-d(x,z_i)
=
h_{z_i}(y)-h_{z_i}(x)
\to
h_x(y)-h_x(x)
=
d(y,x).
\]
Hence the condition in \textup{(i)} is equivalent to pointwise convergence \(h_{z_i}\to h_x\).

Since \(h_{x_0}: X\to h_{x_0}(X)\) is continuous and injective, \(X\) is Gromov-compactifiable if and only if the inverse
\[
h_{x_0}^{-1}: h_{x_0}(X)\to X
\]
is continuous. By the net criterion for continuity, this is equivalent to
\[
h_{z_i}\to h_x
\quad\text{implies}\quad
z_i\to x.
\]
Together with the preceding equivalence, this proves \textup{(i)}.

Now suppose that \(X\) is separable. By \cite[Proposition~1.2(ii)]{MR4952998}, \(\overline X^{\rm h}\), and hence \(h_{x_0}(X)\), is metrizable. Therefore \(h_{x_0}^{-1}\) is continuous if and only if it is sequentially continuous. Applying the preceding argument to sequences proves \textup{(ii)}.
\end{proof}

\begin{proposition}[{\cite[Remark~2.2(c)]{MR4952998}}]\label{pro:finite_set_criterion}
A metric space \((X,d)\) is Gromov-compactifiable if and only if for every \(x\in X\) and \(r>0\), there exist a finite set \(F\subset X\) and \(\eta>0\) such that for every \(z\in X\) with \(d(z,x)\geq r\), there is some \(w\in F\) satisfying
\[
d(w,z)\leq d(w,x)+d(x,z)-\eta.
 \]
 \end{proposition}

\subsection{Wasserstein spaces}\label{section2.1}

We recall the notation used below; see
\cite[\S5.1]{MR2401600} and \cite[Chapter~1]{MR2459454}.
Let \((X,d)\) be a Polish metric space, and let \(\mathcal P(X)\) denote
the set of Borel probability measures on \(X\). If \(\lambda\) is a finite
Borel measure on a metric space and \(E\) is Borel, we write
\(\lambda|_E\) for the restriction of \(\lambda\) to \(E\), that is,
\[
(\lambda|_E)(B):=\lambda(B\cap E)
\]
for every Borel set \(B\).

For \(\mu,\nu\in\mathcal P(X)\), let \(\Pi(\mu,\nu)\) be the set of
\(\pi\in\mathcal P(X\times X)\) satisfying
\[
(\rho_1)_\#\pi=\mu,
\qquad
(\rho_2)_\#\pi=\nu,
\]
where \(\rho_1,\rho_2\colon X\times X\to X\) are the coordinate
projections. We shall use the identities
\[
(\rho_1)_\#\big(\pi|_{E\times X}\big)
=
\big((\rho_1)_\#\pi\big)|_E
\]
and, for finite Borel measures \(\alpha,\beta\) on \(X\),
\[
(\rho_1)_\#(\alpha\otimes\beta)
=
\beta(X)\alpha,
\qquad
(\rho_2)_\#(\alpha\otimes\beta)
=
\alpha(X)\beta.
\]

For \(p\in[1,\infty)\), the \(p\)-Wasserstein space is 
\[
\mathcal P_p(X)
:=
\Set{\mu\in\mathcal P(X)}
{\int_X d(x,x_0)^p\,d\mu(x)<\infty
\text{ for some }x_0\in X}.
\]
This definition does not depend on the choice of \(x_0\). The \(p\)-Wasserstein distance is
\[
W_p(\mu, \nu) := \bigg( \inf_{\pi \in \Pi(\mu, \nu)} \int_{X\times X} d(x, y)^p \, d\pi(x, y) \bigg)^{1/p}.
\]
For \(1\leq p<\infty\), the infimum is attained by an optimal coupling; see \cite[Theorem~4.1]{MR2459454}. For \(p=\infty\), let \(\mathcal P_\infty(X)\) consist of those \(\mu\in\mathcal P(X)\) whose support \(\operatorname{supp}(\mu)\) is bounded. For a nonzero finite Borel measure \(\pi\) on \(X\times X\), let \(\|\pi\|_\infty\) be the essential supremum of the function \((x,y)\mapsto d(x,y)\) with respect to \(\pi\). Since \(d\) is continuous,
\[
\|\pi\|_\infty
=
\sup_{(x,y)\in\operatorname{supp}(\pi)}d(x,y).
\]
We set \(\|0\|_\infty:=0\).
The \(\infty\)-Wasserstein distance is
\[
W_\infty(\mu,\nu)
:=
\inf_{\pi\in\Pi(\mu,\nu)}\|\pi\|_\infty;
\]
see, for example, \cite{MR2403310}. The infimum is attained. Indeed,
\(\Pi(\mu,\nu)\) is weakly compact, and for every \(R\ge0\),
\[
\Set{\pi\in\Pi(\mu,\nu)}{\|\pi\|_\infty\leq R}
=
\Set{\pi\in\Pi(\mu,\nu)}{\pi(\{d>R\})=0}
\]
is weakly closed by the Portmanteau theorem, since \(\{d>R\}\) is open.
Thus \(\pi\mapsto\|\pi\|_\infty\) is weakly lower semicontinuous and
attains its minimum on \(\Pi(\mu,\nu)\). 

\begin{proposition}\label{pro:topology_wasserstein}
Let \((X,d)\) be a Polish metric space.
\begin{enumerate}[\rm (i)]
    \item For every \(p\in[1,\infty)\), the Wasserstein space \((\mathcal P_p(X),W_p)\) is a Polish metric space.
    \item If \(X\) contains at least two points, the \(\infty\)-Wasserstein space \((\mathcal P_\infty(X),W_\infty)\) is not separable.
\end{enumerate}
\end{proposition}

\begin{proof}
Part (i) is standard; see \cite[Theorem~6.18]{MR2459454}.

For (ii), choose distinct points \(x_1,x_2\in X\). For \(t\in[0,1]\), let
\[
\mu_t:=(1-t)\delta_{x_1}+t\delta_{x_2}\in\mathcal P_\infty(X).
\]
If \(0\leq s<t\le1\) and \(\pi\in\Pi(\mu_s,\mu_t)\), then the marginal conditions give
\[
    \pi(\{x_1\} \times \{x_2\}) = t - \pi(\{x_2\} \times \{x_2\}).
\]
Since
\[
\pi(\{x_2\}\times\{x_2\})
\le
\pi(\{x_2\}\times X)
=
s,
\]
we have
\[
\pi(\{x_1\}\times\{x_2\})\geq t-s>0.
\]
Thus \((x_1,x_2)\in\operatorname{supp}(\pi)\), so \(\|\pi\|_\infty\geq d(x_1,x_2)\). Taking the infimum over all couplings \(\pi\in \Pi(\mu_s,\mu_t)\) gives \(W_\infty(\mu_s,\mu_t)\geq d(x_1,x_2)>0\) for all \(0\leq s<t\le1\). Hence \(\mathcal P_\infty(X)\) is not separable.
\end{proof}

\section{The case \texorpdfstring{\(p=1\)}{p=1}}\label{S2}

The necessity of Theorem~\ref{thm:p1_compactifiable} follows from \cite[Proposition~3.16(ii)]{MR4952998} and \cite[Theorem~1.1]{MR5096390}. We prove the converse using tightness and a uniform form of Proposition~\ref{pro:finite_set_criterion} on compact subsets of \(X\).

Recall that a probability measure \(\mu\) on a metric space \(X\) is tight if, for every \(\varepsilon>0\), there exists a compact set \(K\subset X\) such that \(\mu(K)>1-\varepsilon\).

\begin{lemma}[{\cite[Theorem~1.4]{MR233396}}]\label{lem1_2}
If \((X,d)\) is a Polish metric space, then every Borel probability measure on \(X\) is tight.
\end{lemma}

For \(w,x,z\in X\), set
\[
E_w(x,z):=d(w,x)+d(x,z)-d(w,z)\ge0.
\]
Since the only coupling of \(\delta_w\) and \(\lambda\in\mathcal P_1(X)\) is \(\delta_w\otimes\lambda\),
\[
W_1(\delta_w,\lambda)=\int_X d(w,x)\,d\lambda(x).
\]
Thus, if \(\pi\in\Pi(\mu,\nu)\) is optimal, then
\begin{align}
\Delta_w(\mu,\nu)
&:=
W_1(\delta_w,\mu)+W_1(\mu,\nu)-W_1(\delta_w,\nu)\nonumber\\
&=
\int_{X\times X}
\big(d(w,x)+d(x,z)-d(w,z)\big)\,d\pi(x,z)\nonumber\\
&=
\int_{X\times X}E_w(x,z)\,d\pi(x,z).
\label{eq:p1_defect_identity}
\end{align}

\begin{lemma}\label{lem:uniform_defect}
Let \(X\) be Gromov-compactifiable, let \(K\subset X\) be a nonempty compact set, and let \(\rho>0\). Then there exist a nonempty finite set \(F\subset X\) and \(\eta>0\) such that
\[
\max_{w\in F}E_w(x,z)\ge\eta
\]
whenever \(x\in K\) and \(d(x,z)\ge\rho\).
\end{lemma}

\begin{proof}
For each \(u\in K\), Proposition~\ref{pro:finite_set_criterion}, applied at \(u\) with radius \(\rho/2\), gives a nonempty finite set \(F_u\subset X\) and \(\eta_u>0\) such that
\[
\max_{w\in F_u}E_w(u,z)\ge\eta_u
\qquad\text{whenever }d(u,z)\ge\frac{\rho}{2}.
\]
Set \(a_u:=\frac14\min\{\rho,\eta_u\}\). For \(w,u,x,z\in X\),
\[
|E_w(x,z)-E_w(u,z)|
\le
|d(w,x)-d(w,u)|+|d(x,z)-d(u,z)|
\le
2d(x,u).
\]
Hence, if \(d(x,u)<a_u\) and \(d(x,z)\ge\rho\), then
\[
d(u,z)
\ge
d(x,z)-d(x,u)
>
\rho-a_u
\ge
\frac{3\rho}{4}
>
\frac{\rho}{2},
\]
and therefore
\[
\max_{w\in F_u}E_w(x,z)
\ge
\eta_u-2d(x,u)
>
\frac{\eta_u}{2}.
\]

By compactness of \(K\), choose \(u_1,\ldots,u_N\in K\) such that
\[
K\subset\bigcup_{i=1}^N B(u_i,a_{u_i}),
\]
and set
\[
F:=\bigcup_{i=1}^N F_{u_i},
\qquad
\eta:=\frac12\min_{1\leq i\leq N}\eta_{u_i}.
\]
Then \(\max_{w\in F}E_w(x,z)\ge\eta\) whenever \(x\in K\) and \(d(x,z)\ge\rho\).
\end{proof}

\begin{proof}[Proof of Theorem~\ref{thm:p1_compactifiable}]
Let \(X\) be bounded and Gromov-compactifiable. The singleton case is trivial, so set
\[
M:=\operatorname{diam}(X)=\sup\Set{d(x,y)}{x,y\in X}>0.
\]
Fix \(\mu\in\mathcal P_1(X)\) and \(r>0\). Since
\[
W_1(\mu,\nu)\leq M
\qquad
\text{for all }\nu\in\mathcal P_1(X),
\]
it is enough to consider \(0<r\leq M\).

By Lemma~\ref{lem1_2}, choose a compact set \(K\subset X\) such that
\[
\mu(X\setminus K)<\frac{r}{4M}.
\]
Apply Lemma~\ref{lem:uniform_defect} to \(K\) with \(\rho=r/2\), and let \(F\subset X\) and \(\eta>0\) be as in the lemma.

Let \(\nu\in\mathcal P_1(X)\) satisfy \(W_1(\mu,\nu)\geq r\), and let \(\pi\in\Pi(\mu,\nu)\) be optimal. Set
\[
A:=\Set{(x,z)\in X\times X}{d(x,z)\ge r/2},
\qquad
B:=A\cap(K\times X).
\]
Since \(d\leq M\), we obtain
\begin{align*}
r
 \le
\int_{X\times X}d(x,z)\,d\pi(x,z)\le
\frac r2\,\pi(A^c)+M\pi(A)\le
\frac r2+M\pi(A).
\end{align*}
Hence
\[
\pi(A)\ge\frac{r}{2M}.
\]
Using the first marginal of \(\pi\),
\begin{align*}
\pi(B)
\ge
\pi(A)-\pi((X\setminus K)\times X)=
\pi(A)-\mu(X\setminus K)>
\frac{r}{4M}.
\end{align*}

For every \((x,z)\in B\),
\[
\sum_{w\in F}E_w(x,z)
\ge
\max_{w\in F}E_w(x,z)
\ge
\eta.
\]
Therefore, by \eqref{eq:p1_defect_identity},
\begin{align*}
\sum_{w\in F}\Delta_w(\mu,\nu)
&=
\int_{X\times X}\sum_{w\in F}E_w(x,z)\,d\pi(x,z)\\
&\ge
\int_B\sum_{w\in F}E_w(x,z)\,d\pi(x,z)\ge
\eta\pi(B)>
\frac{\eta r}{4M}.
\end{align*}
Set \(\varepsilon:=\frac{\eta r}{4M|F|}>0\), where \(|F|\) is the cardinality of \(F\). There is therefore some \(w\in F\) such that \(\Delta_w(\mu,\nu)>\varepsilon\), and hence
\[
W_1(\delta_w,\nu)
<
W_1(\delta_w,\mu)+W_1(\mu,\nu)-\varepsilon.
\]
Thus the set \(\Set{\delta_w}{w\in F}\subset\mathcal P_1(X)\) satisfies Proposition~\ref{pro:finite_set_criterion} at \(\mu\in\mathcal P_1(X)\) and radius \(r>0\). Since \(\mu\in\mathcal P_1(X)\) and \(r>0\) were arbitrary, \(\mathcal P_1(X)\) is Gromov-compactifiable.
\end{proof}

\section{The range \texorpdfstring{\(1<p<\infty\)}{1 < p < infinity}}\label{S3}

The argument for \(1<p<\infty\) uses the uniform convexity of \(L^p\), in the form of Clarkson's inequalities.

\begin{lemma}[Clarkson's inequalities, \cite{MR2759829}]\label{lem:clarkson}
If \(2\leq p<\infty\), then for all \(f,g\in L^p\),
\[
\left\|\frac{f+g}{2}\right\|_p^p
+
\left\|\frac{f-g}{2}\right\|_p^p
\le
\frac12\big(\|f\|_p^p+\|g\|_p^p\big).
\]
If \(1<p\le2\), then for all \(f,g\in L^p\),
\[
\left\|\frac{f+g}{2}\right\|_p^q
+
\left\|\frac{f-g}{2}\right\|_p^q
\le
\left(
\frac12\|f\|_p^p+\frac12\|g\|_p^p
\right)^{1/(p-1)},
\]
where \(q=p/(p-1)\) is the conjugate exponent of \(p\).
\end{lemma}

\begin{lemma}\label{lem:asymptotic_dirac}
Let \(X\) be a Polish metric space and let \(p\in(1,\infty)\). Suppose that
\(\mu^*\in\mathcal P_p(X)\) and \((\mu_n)\subset\mathcal P_p(X)\) satisfy
\[
\liminf_{n\to\infty}W_p(\mu^*,\mu_n)>0
\]
and
\begin{equation}\label{eq:W_p_convergence}
W_p(\delta_y,\mu_n)-W_p(\mu^*,\mu_n)
\to
W_p(\delta_y,\mu^*)
\qquad
\text{for every }y\in X.
\end{equation}
Then \(\mu^*\) is a Dirac measure.
\end{lemma}

\begin{proof}
Suppose that \(\mu^*\) is not a Dirac measure. Choose distinct
\(y_1,y_2\in\operatorname{supp}(\mu^*)\), and set
\[
A_j:=W_p(\delta_{y_j},\mu^*)>0
\quad (j=1,2),
\qquad
c_n:=W_p(\mu^*,\mu_n).
\]
After discarding finitely many terms, we may assume that \(c_n\ge\delta>0\) for every \(n\).

For each \(n\), let \(\pi_n\in\Pi(\mu^*,\mu_n)\) be optimal, and define on \(X\times X\)
\[
F_j(x,z):=d(y_j,x),
\qquad
G(x,z):=d(x,z).
\]
Then
\[
\|F_j\|_{L^p(\pi_n)}=A_j,
\qquad
\|G\|_{L^p(\pi_n)}=c_n.
\]
Moreover, the triangle inequality gives
\[
d(y_j,z)\leq F_j(x,z)+G(x,z),
\]
and therefore
\[
W_p(\delta_{y_j},\mu_n)
\leq
\|F_j+G\|_{L^p(\pi_n)}
\leq
A_j+c_n.
\]
It follows from \eqref{eq:W_p_convergence} that
\begin{equation}\label{eq:defect_vanish}
A_j+c_n-\|F_j+G\|_{L^p(\pi_n)}
\to0
\qquad (j=1,2).
\end{equation}

Set
\[
U_j:=\frac{F_j}{A_j},
\qquad
V_n:=\frac{G}{c_n},
\qquad
m_j:=\min\{A_j,\delta\}>0.
\]
Thus
\[
\|U_j\|_{L^p(\pi_n)}
=
\|V_n\|_{L^p(\pi_n)}
=
1.
\]
Since
\[
F_j+G
=
m_j(U_j+V_n)
+
(A_j-m_j)U_j
+
(c_n-m_j)V_n,
\]
the triangle inequality and \eqref{eq:defect_vanish} yield
\[
0
\le
m_j\big(2-\|U_j+V_n\|_{L^p(\pi_n)}\big)
\le
A_j+c_n-\|F_j+G\|_{L^p(\pi_n)}
\to0.
\]
Therefore
\[
\|U_j+V_n\|_{L^p(\pi_n)}\to2
\qquad (j=1,2).
\]

Let
\[
s:=
\begin{cases}
\dfrac{p}{p-1}, & 1<p<2,\\[1ex]
p, & 2\leq p<\infty.
\end{cases}
\]
Lemma \ref{lem:clarkson} yields
\[
0
\le
\left\|\frac{U_j-V_n}{2}\right\|_{L^p(\pi_n)}^s
\le
1-
\left\|\frac{U_j+V_n}{2}\right\|_{L^p(\pi_n)}^s
\to0,
\]
so
\[
\|U_j-V_n\|_{L^p(\pi_n)}\to0
\qquad (j=1,2).
\]
Consequently,
\[
\|U_1-U_2\|_{L^p(\pi_n)}
\le
\|U_1-V_n\|_{L^p(\pi_n)}
+
\|V_n-U_2\|_{L^p(\pi_n)}
\to0.
\]
Since the first marginal of \(\pi_n\) is \(\mu^*\),
\[
\int_X
\left|
\frac{d(y_1,x)}{A_1}
-
\frac{d(y_2,x)}{A_2}
\right|^p
\,d\mu^*(x)
=
\|U_1-U_2\|_{L^p(\pi_n)}^p
\to0.
\]
The integral on the left does not depend on \(n\), so it is zero. Its integrand is continuous, and every neighborhood of a point in \(\operatorname{supp}(\mu^*)\) has positive \(\mu^*\)-measure. Hence
\[
\frac{d(y_1,x)}{A_1}
=
\frac{d(y_2,x)}{A_2}
\qquad
\text{for every }x\in\operatorname{supp}(\mu^*).
\]
Taking \(x=y_1\) gives \(0=\frac{d(y_2,y_1)}{A_2}\), contrary to \(y_1\ne y_2\). Thus \(\mu^*\) is a Dirac measure.
\end{proof}

\begin{proof}[Proof of Theorem~\ref{thm:compactifiable_P_p}]
The assertion is immediate if \(X\) is a singleton. Assume otherwise.

Let \((\mu_n)\subset\mathcal P_p(X)\) and \(\mu^*\in\mathcal P_p(X)\) satisfy
\begin{equation}\label{eq:horo_seq}
W_p(\nu,\mu_n)-W_p(\mu^*,\mu_n)
\to
W_p(\nu,\mu^*)
\qquad
\text{for every }\nu\in\mathcal P_p(X).
\end{equation}
Set \(c_n:=W_p(\mu^*,\mu_n)\). By Proposition~\ref{pro:topology_wasserstein}(i),
\(\mathcal P_p(X)\) is Polish and hence separable. In view of
Proposition~\ref{pro:net_convergence}(ii), it remains to prove that \(c_n\to0\).

Suppose otherwise. Passing to a subsequence and relabeling, there is
\(\delta>0\) such that \(c_n\ge\delta\) for every \(n\). Taking \(\nu=\delta_y\) in \eqref{eq:horo_seq} and applying
Lemma~\ref{lem:asymptotic_dirac}, we obtain \(\mu^*=\delta_{x^*}\) for some \(x^*\in X\).

Choose \(y\ne x^*\), and set
\[
D:=d(x^*,y)>0,
\qquad
\nu:=\frac12\delta_{x^*}+\frac12\delta_y.
\]
The triangle inequality gives
\[
W_p(\delta_y,\mu_n)
\le
W_p(\delta_y,\delta_{x^*})
+
W_p(\delta_{x^*},\mu_n)
=
D+c_n.
\]
Using the coupling
\[
\frac12\,\delta_{x^*}\otimes\mu_n
+
\frac12\,\delta_y\otimes\mu_n
\in\Pi(\nu,\mu_n),
\]
we obtain
\begin{equation}\label{eq:Wp_nu_mun}
W_p^p(\nu,\mu_n)
\le
\frac12c_n^p
+
\frac12W_p^p(\delta_y,\mu_n)
\le
\frac12c_n^p+\frac12(c_n+D)^p.
\end{equation}
Define the function
\[
f(t):=
\left(
\frac{t^p+(t+D)^p}{2}
\right)^{1/p}
-t,
\qquad t\ge0.
\]
It follows from \eqref{eq:Wp_nu_mun} that
\[
W_p(\nu,\mu_n)-c_n\leq f(c_n).
\]
We claim that \(f\) is strictly decreasing. Indeed,
\[
f'(t)
=
\frac{\big(t^{p-1}+(t+D)^{p-1}\big)/2}{
\left(
\dfrac{t^p+(t+D)^p}{2}
\right)^{(p-1)/p}
}
-1.
\]
Since \(D>0\), Hölder's inequality gives
\[
\frac{t^{p-1}+(t+D)^{p-1}}{2}
<
\left(
\frac{t^p+(t+D)^p}{2}
\right)^{(p-1)/p},
\]
so \(f'(t)<0\) for all \(t\geq0\). Since \(c_n\geq\delta\),
\[
W_p(\nu,\mu_n)-c_n
\leq f(c_n)
\leq f(\delta)
<f(0)
=2^{-1/p}D
=W_p(\nu,\delta_{x^*}).
\] 
Letting \(n\to\infty\) contradicts \eqref{eq:horo_seq}. Thus \(c_n\to0\). Proposition~\ref{pro:net_convergence}(ii) now implies that
\(\mathcal P_p(X)\) is Gromov-compactifiable.
\end{proof}

\section{The case \texorpdfstring{\(p=\infty\)}{p = infinity}}\label{S4}

Unless \(X\) is a singleton, Proposition~\ref{pro:topology_wasserstein}(ii) shows that \((\mathcal P_\infty(X),W_\infty)\) is nonseparable. We therefore work with the finite-set criterion. In the sufficiency argument, each test point of \(X\) gives a measure obtained by moving a small amount of mass to that point. Lemma~\ref{lem:radial_mass_bound} provides the gap in transport radius needed to use these measures as witnesses.

\subsection{Necessity}

\begin{proof}[Proof of necessity in Theorem~\ref{thm3}]
Fix \(x\in X\) and \(r>0\). If \(d(x,z)<r\) for every \(z\in X\), then the finite-set criterion for \(X\) at \(x\) and \(r\) is vacuous, and there is nothing to prove. Thus choose \(z_0\in X\) with \(d(x,z_0)\geq r\).

By Proposition~\ref{pro:finite_set_criterion}, there are
\(\sigma_1,\ldots,\sigma_k\in\mathcal P_\infty(X)\) and \(\eta>0\) such that, whenever \(W_\infty(\nu,\delta_x)\geq r\), some \(i\) satisfies
\[
W_\infty(\sigma_i,\nu)
\le
W_\infty(\sigma_i,\delta_x)
+
W_\infty(\delta_x,\nu)
-\eta.
\]
Taking \(\nu=\delta_{z_0}\) shows that \(k\ge1\).

For each \(i\),
\[
W_\infty(\sigma_i,\delta_x)
=
\sup_{y\in\operatorname{supp}(\sigma_i)}d(y,x).
\]
Choose \(w_i\in\operatorname{supp}(\sigma_i)\) such that
\[
d(w_i,x)
\ge
W_\infty(\sigma_i,\delta_x)-\frac{\eta}{2}.
\]
If \(d(z,x)\geq r\), then for some \(i\),
\begin{align*}
d(w_i,z)
&\le
W_\infty(\sigma_i,\delta_z)\\
&\le
W_\infty(\sigma_i,\delta_x)+d(x,z)-\eta\\
&\le
d(w_i,x)+d(x,z)-\frac{\eta}{2}.
\end{align*}
Proposition~\ref{pro:finite_set_criterion}, with \(F:=\{w_1,\ldots,w_k\}\) and \(\eta/2\), gives the result.
\end{proof}

\subsection{Sufficiency}

For \(w\) away from \(x\), we replace the mass near \(w\), together with a small portion of the mass near \(x\), by an atom at \(w\). We shall use the following estimate. As usual, \(B(x,r)\) and \(\overline B(x,r)\) denote the open and closed balls centered at \(x\).

\begin{lemma}\label{lem:radial_mass_bound}
Let \(\mu\in\mathcal P_\infty(X)\), let \(x,w\in X\), and let \(\rho>0\). Set
\[
A:=B(x,\rho),
\qquad
R:=d(x,w),
\]
and assume that \(R>\rho\) and \(\mu(A)>0\). For \(q\in(0,\mu(A))\), set
\[
C:=\overline B(w,R-\rho),\qquad 
\sigma
:=
\mu-\mu|_C-\frac{q}{\mu(A)}\mu|_A
+\big(\mu(C)+q\big)\delta_w.
\]
Then \(\sigma\in\mathcal P_\infty(X)\), and
\[
R-\rho<W_\infty(\sigma,\mu)\leq R+\rho.
\]
\end{lemma}

\begin{proof}
For \(y\in A\),
\[
d(w,y)\geq R-d(x,y)>R-\rho,
\]
so \(A\cap C=\varnothing\). Hence
\[
\sigma
=
\mu|_{X\setminus(A\cup C)}
+
\left(1-\frac{q}{\mu(A)}\right)\mu|_A
+
\big(\mu(C)+q\big)\delta_w,
\]
and therefore \(\sigma\in\mathcal P_\infty(X)\).

Moving \(\mu|_C+\frac{q}{\mu(A)}\mu|_A\) to \(w\) and leaving the remaining mass fixed gives
\[
W_\infty(\sigma,\mu)\leq R+\rho.
\]

Let \(\gamma\in\Pi(\mu,\sigma)\) be optimal. If \(W_\infty(\sigma,\mu)\leq R-\rho\), then all mass transported to \(w\) must originate in \(C\). Since \(w\in C\),
\[
\mu(C)+q
=
\sigma(\{w\})
=
\gamma(X\times\{w\})
=
\gamma(C\times\{w\})
\le
\gamma(C\times X)
=
\mu(C),
\]
a contradiction. Hence
\[
W_\infty(\sigma,\mu)>R-\rho.
\] 
\end{proof}

\begin{proof}[Proof of sufficiency in Theorem~\ref{thm3}]
Fix \(\mu\in\mathcal P_\infty(X)\) and \(r>0\), and choose
\(x\in\operatorname{supp}(\mu)\). If \(d(x,z)<\frac r3\) for every \(z\in X\),  then
\[
W_\infty(\lambda,\lambda')
\le
\operatorname{diam}(X)
\le
\frac{2r}{3}
<
r
\qquad\text{for all }\lambda,\lambda'\in\mathcal P_\infty(X),
\]
so there is nothing to prove. We may assume that \(d(x,z)\geq \frac{r}{3}\) for some \(z\in X\). 

By Proposition~\ref{pro:finite_set_criterion}, there are
\(w_1,\ldots,w_m\in X\) and \(\varepsilon>0\) such that, for every
\(z\in X\) with \(d(x,z)\geq r/3\), there is some \(k\in\{1,\ldots,m\}\) for which
\begin{equation}\label{eq:pinf_point_witness}
d(w_k,z)
\leq
d(w_k,x)+d(x,z)-\varepsilon.
\end{equation}
A point \(w_k=x\) never satisfies \eqref{eq:pinf_point_witness}, so such points may be omitted. By our assumption above, the remaining family is nonempty. Set
\[
R_k:=d(x,w_k)>0,
\qquad
1\leq k\leq m.
\]

Choose
\[
0<\rho<
\frac14
\min\left\{\frac r3,\varepsilon,R_1,\ldots,R_m\right\},
\qquad
A:=B(x,\rho).
\]
For \(1\leq k\leq m\), let
\[
C_k:=\overline B(w_k,R_k-\rho)
\]
and
\[
\sigma_k
:=
\mu-\mu|_{C_k}
-\frac1{m+1}\mu|_A
+\Big(\mu(C_k)+\frac{\mu(A)}{m+1}\Big)\delta_{w_k}.
\]
By Lemma~\ref{lem:radial_mass_bound}, \(\sigma_k\in\mathcal P_\infty(X)\), and
\begin{equation}\label{eq:pinf_radial_bounds}
R_k-\rho<W_\infty(\sigma_k,\mu)\leq R_k+\rho.
\end{equation}

Since the family is finite, we may choose \(\eta>0\) such that
\[
0<\eta<
\min_{1\leq k\leq m}
\{W_\infty(\sigma_k,\mu) -R_k+\rho,\ W_\infty(\sigma_k,\mu) -2\rho\}.
\]
Indeed, \eqref{eq:pinf_radial_bounds} gives
\[
W_\infty(\sigma_k,\mu) -R_k+\rho>0,
\qquad
W_\infty(\sigma_k,\mu) -2\rho
>
R_k-3\rho
>
0.
\]

Let \(\nu\in\mathcal P_\infty(X)\) satisfy \(W_\infty(\mu,\nu)\geq r\), and let \(\pi\in\Pi(\mu,\nu)\) be optimal, i.e., \(\|\pi\|_\infty=W_\infty(\mu,\nu)\).
Set
\[
H_0
:=
\Set{(y,z)\in A\times X}
{d(y,z)\leq W_\infty(\mu,\nu)-\frac r3}
\]
and, for \(1\leq k\leq m\),
\[
H_k
:=
\Set{(y,z)\in A\times X}
{d(w_k,z)\leq R_k+d(x,z)-\varepsilon}.
\]
If \((y,z)\in(A\times X)\setminus H_0\), then
\[
d(x,z)
\ge
d(y,z)-d(x,y)
>
W_\infty(\mu,\nu) -\frac r3-\rho
\ge
\frac{2r}{3}-\rho
>
\frac r3.
\]
Hence \eqref{eq:pinf_point_witness} applies and
\[
A\times X=\bigcup_{j=0}^mH_j.
\]
Therefore
\[
\mu(A)
=
\pi(A\times X)
\le
\sum_{j=0}^m\pi(H_j),
\]
so for some \(j\in\{0,\ldots,m\}\), \(\pi(H_j)\geq \frac{\mu(A)}{m+1}\). Set
\[
\theta:=\frac{\mu(A)}{(m+1)\pi(H_j)}\,\pi|_{H_j}.
\]
Then
\[
0\le\theta\le\pi|_{A\times X},
\qquad
\theta(X\times X)=\frac{\mu(A)}{m+1}.
\]
If \(j=0\), take \(k=1\); if \(j>0\), take \(k=j\).

It follows from the triangle inequality that \(A\cap C_k=\varnothing\). Write
\[
\pi_A:=\pi|_{A\times X},
\qquad
\pi_C:=\pi|_{C_k\times X},
\qquad
\pi_O:=\pi|_{(X\setminus(A\cup C_k))\times X}.
\]
Then \(\pi=\pi_A+\pi_C+\pi_O\). Define
\[
\xi:=(\rho_2)_\#(\pi_C+\theta),
\qquad
\tau:=(\rho_2)_\#(\pi_A-\theta).
\]
Since
\[
\pi_C(X\times X)=\mu(C_k),
\qquad
\pi_A(X\times X)=\mu(A),
\]
we have
\[
\xi(X)=\mu(C_k)+\frac{\mu(A)}{m+1},
\qquad
\tau(X)=\mu(A)-\frac{\mu(A)}{m+1}=\frac{m}{m+1}\mu(A),
\]
and
\[
\xi+\tau+(\rho_2)_\#\pi_O
=
(\rho_2)_\#\pi
=
\nu.
\]

Set
\[
\Gamma
:=
\delta_{w_k}\otimes\xi
+
\frac1{\mu(A)}\,\mu|_A\otimes\tau
+
\pi_O.
\]
Using the elementary marginal identities in Subsection \ref{section2.1},
\[
(\rho_1)_\#\Gamma
=
\Big(\mu(C_k)+\frac{\mu(A)}{m+1}\Big)\delta_{w_k}
+
\frac{m}{m+1}\mu|_A
+
\mu|_{X\setminus(A\cup C_k)}
=
\sigma_k,
\]
while
\[
(\rho_2)_\#\Gamma
=
\xi+\tau+(\rho_2)_\#\pi_O
=
\nu.
\]
Thus \(\Gamma\in\Pi(\sigma_k,\nu)\).

It remains to estimate its transport radius. For
\(\pi_C\)-almost every \((y,z)\),
\begin{align*}
d(w_k,z)
&\le
d(w_k,y)+d(y,z)\\ &
\le
R_k-\rho+W_\infty(\mu,\nu)\\
&<
W_\infty(\sigma_k,\mu) +W_\infty(\mu,\nu)-\eta.
\end{align*}
For \(\theta\)-almost every \((y,z)\), suppose first that \(j=0\). Then
\[
\begin{aligned}
d(w_k,z)
&\le
R_k+d(x,y)+d(y,z)\\
&<
R_k+\rho+W_\infty(\mu,\nu)-\frac r3\\
&=
R_k-\rho+W_\infty(\mu,\nu)-\left(\frac r3-2\rho\right)\\
&<
W_\infty(\sigma_k,\mu) +W_\infty(\mu,\nu) -\eta.
\end{aligned}
\]
If \(j>0\), then \(k=j\), and
\[
\begin{aligned}
d(w_k,z)
&\le
R_k+d(x,z)-\varepsilon\\
&\le
R_k+d(x,y)+d(y,z)-\varepsilon\\
&<
R_k+\rho+W_\infty(\mu,\nu) -\varepsilon\\
&=
R_k-\rho+W_\infty(\mu,\nu) -(\varepsilon-2\rho)\\
&<
W_\infty(\sigma_k,\mu) +W_\infty(\mu,\nu) -\eta.
\end{aligned}
\]
Hence
\begin{equation}\label{eq:comp_1}
\|\delta_{w_k}\otimes\xi\|_\infty
\le
W_\infty(\sigma_k,\mu) +W_\infty(\mu,\nu)-\eta.
\end{equation}

Since \(0\le\pi_A-\theta\le\pi_A\) and \(\|\pi\|_\infty=W_\infty(\mu,\nu)\), the measure \(\tau\) is concentrated on
\(\overline B(x,W_\infty(\mu,\nu)+\rho)\). Thus
\begin{equation}\label{eq:comp_2}
\left\|
\frac1{\mu(A)}\,\mu|_A\otimes\tau
\right\|_\infty
\le
W_\infty(\mu,\nu) +2\rho
<
W_\infty(\sigma_k,\mu) +W_\infty(\mu,\nu)-\eta.
\end{equation}
Finally,
\begin{equation}\label{eq:comp_3}
\|\pi_O\|_\infty
\le
W_\infty(\mu,\nu)
<
W_\infty(\sigma_k,\mu) +W_\infty(\mu,\nu)-\eta.
\end{equation}
Combining \eqref{eq:comp_1}--\eqref{eq:comp_3}, we conclude 
\[
W_\infty(\sigma_k,\nu)
 \le
\|\Gamma\|_\infty 
 \le
W_\infty(\sigma_k,\mu) +W_\infty(\mu,\nu) -\eta .
\]
Thus the finite family \(\{\sigma_1,\ldots,\sigma_m\}\) satisfies Proposition~\ref{pro:finite_set_criterion} at \(\mu\).
\end{proof}

\section{Gromov-compactifiability of
\texorpdfstring{\(L^\infty\)}{L-infinity} spaces}\label{S5}

For a measure space \((\Omega,\mathcal B,\mu)\), we write
\(L^\infty(\Omega)\) for \(L^\infty(\Omega,\mathcal B,\mu)\) and use
real scalars throughout this section. For \(E\in\mathcal B\), let
\(\chi_E\) denote its indicator function,
\[
\chi_E(t):=
\begin{cases}
1, & t\in E,\\
0, & t\notin E.
\end{cases}
\]
Recall that \(A\in\mathcal B\) is an atom if \(\mu(A)>0\) and every
measurable \(E\subset A\) satisfies
\[
\mu(E)=0
\qquad\text{or}\qquad
\mu(A\setminus E)=0.
\]
The measure \(\mu\) is called atomless if it has no atoms. We use the
convention
\[
\operatorname{sgn}(a):=
\begin{cases}
1, & a\ge0,\\
-1, & a<0,
\end{cases}
\]
and apply \(\operatorname{sgn}\) pointwise to measurable functions.

\begin{theorem}\label{thm4}
Let \((\Omega,\mathcal B,\mu)\) be a measure space. Then
\(L^\infty(\Omega)\) is Gromov-compactifiable if and only if
\[
\mu(\Omega)=0
\qquad\text{or}\qquad
\mu\text{ has an atom}.
\]
\end{theorem}

We shall use the following elementary facts.

\begin{lemma}\label{lem:atoms}
Let \((\Omega,\mathcal B,\mu)\) be a measure space.
\begin{enumerate}[\rm (i)]
\item If \(A\) is an atom and \(f\in L^\infty(\Omega)\), then there is
\(c_f\in\mathbb R\) such that
\[
f=c_f\quad\text{a.e. on }A,
\qquad
|c_f|\le\|f\|_\infty.
\]

\item If \(\mu\) is atomless and \(A_1,\ldots,A_k\in\mathcal B\) satisfy
\[
\mu(A_i)>0
\qquad(1\leq i\leq k),
\]
then there are pairwise disjoint measurable sets \(B_1,\ldots,B_k\)
such that
\[
B_i\subset A_i,
\qquad
\mu(B_i)>0
\qquad(1\leq i\leq k).
\]
\end{enumerate}
\end{lemma}

\begin{proof}
For \textup{(i)}, set
\[
a:=\operatorname*{ess\,inf}_{A}f,
\qquad
b:=\operatorname*{ess\,sup}_{A}f.
\]
If \(a<b\), choose \(c\in(a,b)\). Then
\[
\mu\big(A\cap\{f<c\}\big)>0,
\qquad
\mu\big(A\cap\{f>c\}\big)>0,
\]
which is impossible since \(A\) is an atom. Hence \(a=b=:c_f\), and
\textup{(i)} follows.

For \textup{(ii)}, argue by induction on \(k\). Suppose that
\(B_1,\ldots,B_{k-1}\) have already been chosen. If
\[
\mu\bigg(A_k\setminus\bigcup_{i<k}B_i\bigg)>0,
\]
take
\[
B_k:=A_k\setminus\bigcup_{i<k}B_i.
\]
Otherwise, there exists some \(j<k\) such that \(\mu(A_k\cap B_j)>0\). Since \(\mu\) is atomless, there is a measurable
\(C\subset A_k\cap B_j\) such that
\[
\mu(C)>0,
\qquad
\mu\big((A_k\cap B_j)\setminus C\big)>0.
\]
Replace \(B_j\) by \(B_j\setminus C\) and set \(B_k:=C\). The resulting
sets still have positive measure and are pairwise disjoint.
\end{proof}

\begin{remark}
For localizable measure spaces with \(\mu(\Omega)>0\), it is known that
\(L^\infty(\Omega)\) is octahedral if and only if \(\mu\) is atomless;
see \cite[Corollary~6.7]{CiaciLangemetsLissitsin2022}. The localizability
hypothesis is not needed for this equivalence. If \(\mu\) is atomless, then \(L^\infty(\Omega)\) has the Daugavet
property \cite{Werner2001}, and hence is octahedral
\cite[Corollary~2.5]{BecerraGuerreroLopezPerezRuedaZoca2014}.

Conversely, suppose that \(A\) is an atom and set \(x:=\chi_A\). If
\(\|z\|_\infty=1\), Lemma~\ref{lem:atoms}(i) gives
\[
z=c_z\quad\text{a.e. on }A
\]
for some \(|c_z|\le1\). Therefore
\[
\begin{aligned}
\min\{\|z-x\|_\infty,\|z+x\|_\infty\}
&\le
\|z-\operatorname{sgn}(c_z)x\|_\infty\\
&=
\max\left\{
1-|c_z|,
\,
\|z\chi_{\Omega\setminus A}\|_\infty
\right\}
\le1.
\end{aligned}
\]
Thus \(L^\infty(\Omega)\) is not octahedral by
\cite[Proposition~3.3]{MR4952998}. Hence, whenever \(\mu(\Omega)>0\),
\[
L^\infty(\Omega)\text{ is octahedral}
\quad\text{if and only if}\quad
\mu\text{ is atomless}.
\]
Together with \cite[Theorem~3.4]{MR4952998}, this also yields
Theorem~\ref{thm4}. We give below a proof based only on
Proposition~\ref{pro:finite_set_criterion}.
\end{remark}

\begin{proof}[Proof of Theorem~\ref{thm4}]
If \(\mu(\Omega)=0\), then \(L^\infty(\Omega)=\{0\}\). Assume
\(\mu(\Omega)>0\). Since the norm metric in \(L^\infty(\Omega)\) and the condition in
Proposition~\ref{pro:finite_set_criterion} are invariant under
translations, it is enough to work at the origin.

Suppose first that \(A\) is an atom. Fix \(r>0\) and set
\[
F_r:=\{-r\chi_A,r\chi_A\}.
\]
Let \(z\in L^\infty(\Omega)\) satisfy \(\|z\|_\infty\geq r\). By
Lemma~\ref{lem:atoms}(i),
\[
z=c_z\quad\text{a.e. on }A,
\qquad
|c_z|\le\|z\|_\infty.
\]
Since \(\|\chi_A\|_\infty=1\),
\[
\min_{\varepsilon\in\{-1,1\}}
\|z-\varepsilon r\chi_A\|_\infty
=
\max\left\{
\bigl||c_z|-r\bigr|,
\,
\|z\chi_{\Omega\setminus A}\|_\infty
\right\}
\leq
\|z\|_\infty.
\]
Thus there exists \(w\in F_r\) satisfying
\[
\|w-z\|_\infty
\le
\|z\|_\infty
=
\|w\|_\infty+\|z\|_\infty-r.
\]
Proposition~\ref{pro:finite_set_criterion}, with \(\eta=r\), shows that
\(L^\infty(\Omega)\) is Gromov-compactifiable.

Suppose now that \(\mu\) is atomless. We show that
Proposition~\ref{pro:finite_set_criterion} fails at the origin for
\(r=1\). Let
\[
F:=\{w_1,\ldots,w_k\}\subset L^\infty(\Omega)
\]
be nonempty and finite, and let \(\eta>0\). Choose measurable
representatives of the \(w_i\)'s and set
\[
A_i:=\Set{t\in\Omega}{|w_i(t)|>\|w_i\|_\infty -\eta}\qquad (1\leq i\leq k). 
\]
Then
\[
\mu(A_i)>0
\qquad\text{for all }1\leq i\leq k.
\]
By Lemma~\ref{lem:atoms}(ii), there are pairwise disjoint measurable
sets \(B_i\subset A_i\) with
\[
\mu(B_i)>0
\qquad\text{for all }1\leq i\leq k.
\]
Define
\[
z:=
-\sum_{i=1}^k\operatorname{sgn}(w_i)\chi_{B_i}.
\]
Since the \(B_i\) are pairwise disjoint and \(\mu(B_i)>0\) for every
\(i\),
\[
|z|=\chi_{\bigcup_{i=1}^kB_i}\quad\text{a.e.},
\qquad
\|z\|_\infty=1.
\]
Moreover,
\[
\operatorname*{ess\,sup}_{B_i}|w_i|
> \|w_i\|_\infty -\eta,
\]
and on \(B_i\),
\[
|w_i-z|=|w_i|+1.
\]
Hence, for every \(1\leq i\leq k\),
\[
\begin{aligned}
\|w_i-z\|_\infty
&\ge
1+\operatorname*{ess\,sup}_{B_i}|w_i|\\
&>
\|w_i\|_\infty +1-\eta=
\|w_i\|_\infty+\|z\|_\infty-\eta.
\end{aligned}
\]
Thus \(F\) does not satisfy
Proposition~\ref{pro:finite_set_criterion} at the origin with radius
\(1\). Since \(F\) and \(\eta\) were arbitrary,
\(L^\infty(\Omega)\) is not Gromov-compactifiable.
\end{proof}

\bibliographystyle{amsplain}
\bibliography{References_Wasserstein}

@incollection {MR919829,
    AUTHOR = {Gromov, M.},
     TITLE = {Hyperbolic groups},
 BOOKTITLE = {Essays in group theory},
    SERIES = {Math. Sci. Res. Inst. Publ.},
    VOLUME = {8},
     PAGES = {75--263},
 PUBLISHER = {Springer, New York},
      YEAR = {1987},
      ISBN = {0-387-96618-8},
   MRCLASS = {20F32 (20F06 20F10 22E40 53C20 57R75 58F17)},
  MRNUMBER = {919829},
MRREVIEWER = {Christopher\ W.\ Stark},
       DOI = {10.1007/978-1-4613-9586-7\_3},
       URL = {https://doi.org/10.1007/978-1-4613-9586-7_3},
}

@article {MR4205705,
    AUTHOR = {Ibarluc\'ia, Tom\'as and Megrelishvili, Michael},
     TITLE = {Maximal equivariant compactification of the {U}rysohn spaces
              and other metric structures},
   JOURNAL = {Adv. Math.},
  FJOURNAL = {Advances in Mathematics},
    VOLUME = {380},
      YEAR = {2021},
     PAGES = {Paper No. 107599, 33},
      ISSN = {0001-8708,1090-2082},
   MRCLASS = {54D35 (03C66 22F30 22F50 37B05 46B20 54B30)},
  MRNUMBER = {4205705},
MRREVIEWER = {Konstantin\ L.\ Kozlov},
       DOI = {10.1016/j.aim.2021.107599},
       URL = {https://doi.org/10.1016/j.aim.2021.107599},
}

@incollection {MR4422057,
    AUTHOR = {Karlsson, Anders},
     TITLE = {Elements of a metric spectral theory},
 BOOKTITLE = {Dynamics, geometry, number theory---the impact of {M}argulis
              on modern mathematics},
     PAGES = {276--300},
 PUBLISHER = {Univ. Chicago Press, Chicago, IL},
      YEAR = {2022}
      }

@article {MR4700367,
    AUTHOR = {Arosio, Leandro and Fiacchi, Matteo and Gontard, S\'ebastien
              and Guerini, Lorenzo},
     TITLE = {The horofunction boundary of a {G}romov hyperbolic space},
   JOURNAL = {Math. Ann.},
  FJOURNAL = {Mathematische Annalen},
    VOLUME = {388},
      YEAR = {2024},
    NUMBER = {2},
     PAGES = {1163--1204},
      ISSN = {0025-5831,1432-1807},
   MRCLASS = {32F45 (32H50 53C23)},
  MRNUMBER = {4700367},
MRREVIEWER = {Feng\ Rong},
       DOI = {10.1007/s00208-022-02551-0},
       URL = {https://doi.org/10.1007/s00208-022-02551-0},
}

@article {MR4952998,
    AUTHOR = {Daniilidis, A. and Garrido, M. I. and Jaramillo, J. A. and
              Tapia-Garc\'ia, S.},
     TITLE = {Horofunction extension and metric compactifications},
   JOURNAL = {Trans. Amer. Math. Soc. Ser. B},
  FJOURNAL = {Transactions of the American Mathematical Society. Series B},
    VOLUME = {12},
      YEAR = {2025},
     PAGES = {1130--1155},
      ISSN = {2330-0000},
   MRCLASS = {53C23 (46B03 46B20 51F30 54D35)},
  MRNUMBER = {4952998},
MRREVIEWER = {Johann\ Langemets},
       DOI = {10.1090/btran/234},
       URL = {https://doi.org/10.1090/btran/234},
}

@article {MR5096390,
    AUTHOR = {Jiang, Huajian and Cui, Xiaojun},
     TITLE = {1-{W}asserstein spaces over unbounded {P}olish metric spaces
              are not {G}romov-compactifiable},
   JOURNAL = {Proc. Amer. Math. Soc.},
  FJOURNAL = {Proceedings of the American Mathematical Society},
    VOLUME = {154},
      YEAR = {2026},
    NUMBER = {9},
     PAGES = {4049--4057},
      ISSN = {0002-9939,1088-6826},
   MRCLASS = {54D35 (49Q22 53C23)},
  MRNUMBER = {5096390},
       DOI = {10.1090/proc/17681},
       URL = {https://doi.org/10.1090/proc/17681},
}

@book {MR2401600,
    AUTHOR = {Ambrosio, Luigi and Gigli, Nicola and Savar\'e, Giuseppe},
     TITLE = {Gradient flows in metric spaces and in the space of
              probability measures},
    SERIES = {Lectures in Mathematics ETH Z\"urich},
   EDITION = {Second},
 PUBLISHER = {Birkh\"auser Verlag, Basel},
      YEAR = {2008},
     PAGES = {x+334},
      ISBN = {978-3-7643-8721-1},
   MRCLASS = {49-02 (28A33 35K55 35K90 49Q20 60B05)},
  MRNUMBER = {2401600},
MRREVIEWER = {Pietro\ Celada},
}

@book {MR2459454,
    AUTHOR = {Villani, C\'edric},
     TITLE = {Optimal transport},
    SERIES = {Grundlehren der mathematischen Wissenschaften [Fundamental
              Principles of Mathematical Sciences]},
    VOLUME = {338},
      NOTE = {Old and new},
 PUBLISHER = {Springer-Verlag, Berlin},
      YEAR = {2009},
     PAGES = {xxii+973},
      ISBN = {978-3-540-71049-3},
   MRCLASS = {49-02 (28A75 37J50 49Q20 53C23 58E30)},
  MRNUMBER = {2459454},
MRREVIEWER = {Dario\ Cordero-Erausquin},
       DOI = {10.1007/978-3-540-71050-9},
       URL = {https://doi.org/10.1007/978-3-540-71050-9},
}

@article {MR2403310,
    AUTHOR = {Champion, Thierry and De Pascale, Luigi and Juutinen, Petri},
     TITLE = {The {$\infty$}-{W}asserstein distance: local solutions and
              existence of optimal transport maps},
   JOURNAL = {SIAM J. Math. Anal.},
  FJOURNAL = {SIAM Journal on Mathematical Analysis},
    VOLUME = {40},
      YEAR = {2008},
    NUMBER = {1},
     PAGES = {1--20},
      ISSN = {0036-1410,1095-7154},
   MRCLASS = {49Q20 (49K30)},
  MRNUMBER = {2403310},
       DOI = {10.1137/07069938X},
       URL = {https://doi.org/10.1137/07069938X},
}

@book {MR233396,
    AUTHOR = {Billingsley, Patrick},
     TITLE = {Convergence of probability measures},
 PUBLISHER = {John Wiley \& Sons, Inc., New York-London-Sydney},
      YEAR = {1968},
     PAGES = {xii+253},
   MRCLASS = {60.30},
  MRNUMBER = {233396},
MRREVIEWER = {M.\ M.\ Siddiqui},
}

@book {MR2759829,
    AUTHOR = {Brezis, Haim},
     TITLE = {Functional analysis, {S}obolev spaces and partial differential
              equations},
    SERIES = {Universitext},
 PUBLISHER = {Springer, New York},
      YEAR = {2011},
     PAGES = {xiv+599},
      ISBN = {978-0-387-70913-0},
       DOI = {10.1007/978-0-387-70914-7},
   MRCLASS = {35-01 (46-01 46E35 46N20 47F05)},
  MRNUMBER = {2759829},
MRREVIEWER = {Vicen\c tiu\ D.\ R\u adulescu},
}

@article {CiaciLangemetsLissitsin2022,
    AUTHOR = {Ciaci, Stefano and Langemets, Johann and Lissitsin, Aleksei},
     TITLE = {Attaining strong diameter two property for infinite cardinals},
   JOURNAL = {J. Math. Anal. Appl.},
  FJOURNAL = {Journal of Mathematical Analysis and Applications},
    VOLUME = {513},
      YEAR = {2022},
    NUMBER = {1},
     PAGES = {Paper No. 126185},
       DOI = {10.1016/j.jmaa.2022.126185},
       URL = {https://doi.org/10.1016/j.jmaa.2022.126185},
}

@article {Rieffel2002,
    AUTHOR = {Rieffel, Marc A.},
     TITLE = {Group {$C^*$}-algebras as compact quantum metric spaces},
   JOURNAL = {Doc. Math.},
    VOLUME = {7},
      YEAR = {2002},
     PAGES = {605--651},
       DOI = {10.4171/DM/133},
       URL = {https://doi.org/10.4171/DM/133},
}

@article {Walsh2007,
    AUTHOR = {Walsh, Cormac},
     TITLE = {The horofunction boundary of finite-dimensional normed spaces},
   JOURNAL = {Math. Proc. Cambridge Philos. Soc.},
    VOLUME = {142},
      YEAR = {2007},
    NUMBER = {3},
     PAGES = {497--507},
       DOI = {10.1017/S0305004107000096},
       URL = {https://doi.org/10.1017/S0305004107000096},
}

@article {Gutierrez2019,
    AUTHOR = {Guti\'errez, Armando W.},
     TITLE = {On the metric compactification of infinite-dimensional {$\ell_p$} spaces},
   JOURNAL = {Canad. Math. Bull.},
    VOLUME = {62},
      YEAR = {2019},
    NUMBER = {3},
     PAGES = {491--507},
       DOI = {10.4153/S0008439518000681},
       URL = {https://doi.org/10.4153/S0008439518000681},
}

@article {Gutierrez2020,
    AUTHOR = {Guti\'errez, Armando W.},
     TITLE = {Characterizing the metric compactification of {$L_p$} spaces by random measures},
   JOURNAL = {Ann. Funct. Anal.},
    VOLUME = {11},
      YEAR = {2020},
    NUMBER = {2},
     PAGES = {227--243},
       DOI = {10.1007/s43034-019-00024-1},
       URL = {https://doi.org/10.1007/s43034-019-00024-1},
}

@article {Werner2001,
    AUTHOR = {Werner, Dirk},
     TITLE = {Recent progress on the {D}augavet property},
   JOURNAL = {Irish Math. Soc. Bull.},
    NUMBER = {46},
      YEAR = {2001},
     PAGES = {77--97},
       DOI = {10.33232/BIMS.0046.77.97},
       URL = {https://doi.org/10.33232/BIMS.0046.77.97},
}

@article {BecerraGuerreroLopezPerezRuedaZoca2014,
    AUTHOR = {Becerra Guerrero, Julio and L\'opez-P\'erez, Gin\'es and Rueda Zoca, Abraham},
     TITLE = {Octahedral norms and convex combination of slices in {B}anach spaces},
   JOURNAL = {J. Funct. Anal.},
    VOLUME = {266},
      YEAR = {2014},
    NUMBER = {4},
     PAGES = {2424--2435},
       DOI = {10.1016/j.jfa.2013.09.004},
       URL = {https://doi.org/10.1016/j.jfa.2013.09.004},
}

@article {BertrandKloeckner2012,
    AUTHOR = {Bertrand, J\'er\^ome and Kloeckner, Beno\^it},
     TITLE = {A geometric study of {W}asserstein spaces: {H}adamard spaces},
   JOURNAL = {J. Topol. Anal.},
    VOLUME = {4},
      YEAR = {2012},
    NUMBER = {4},
     PAGES = {515--542},
       DOI = {10.1142/S1793525312500227},
       URL = {https://doi.org/10.1142/S1793525312500227},
}

@article {MR4154570,
    AUTHOR = {Zhu, Guomin and Wu, Hongguang and Cui, Xiaojun},
     TITLE = {Horo-functions associated to atom sequences on the
              {W}asserstein space},
   JOURNAL = {Arch. Math. (Basel)},
    VOLUME = {115},
      YEAR = {2020},
    NUMBER = {5},
     PAGES = {555--566},
  MRNUMBER = {4154570},
       DOI = {10.1007/s00013-020-01490-z},
       URL = {https://doi.org/10.1007/s00013-020-01490-z},
}

@article {MR4249875,
    AUTHOR = {Zhu, Guomin and Li, Wen-Long and Cui, Xiaojun},
     TITLE = {Busemann functions on the {W}asserstein space},
   JOURNAL = {Calc. Var. Partial Differential Equations},
    VOLUME = {60},
      YEAR = {2021},
    NUMBER = {3},
     PAGES = {Paper No. 97, 16},
  MRNUMBER = {4249875},
       DOI = {10.1007/s00526-021-01937-3},
       URL = {https://doi.org/10.1007/s00526-021-01937-3},
}
\end{document}